\documentclass[a4paper, 12pt]{article}
\usepackage{amsmath,amsthm}
\usepackage{amsfonts}
\usepackage{amssymb,latexsym}
\usepackage{enumerate}
\usepackage{color}
\usepackage{array}
\usepackage[all]{xy}
\usepackage{apptools}
\usepackage{pgf}
\usepackage{tikz}
\usepackage{amsrefs}
\usepackage[left=2.5cm,right=2.5cm,top=3.0cm,bottom=3.0cm]{geometry}
\usepackage{hyperref}

\newtheorem{thm}{Theorem}[section]
\newtheorem{cor}[thm]{Corollary}
\newtheorem{lm}[thm]{Lemma}
\newtheorem{prop}[thm]{Proposition}

\newtheorem*{con}{Conjecture}
\newtheorem*{thmintro}{Theorem}

\theoremstyle{definition}
\newtheorem{rmk}[thm]{Remark}
\newtheorem{example}[thm]{Example}

\numberwithin{equation}{section}

\def\NN{{\mathbb{N}}}

\def\CF{{\cal F}}

\def\ra{\rightarrow}

\renewcommand*{\mod}{\mathrm{mod}}
\newcommand{\La}{\Lambda}

\def\id{\mbox{{\rm id}}}

\def\soc{\textnormal{soc}}

\DeclareMathOperator{\T}{T}

\DeclareMathOperator{\ob}{ob}
\DeclareMathOperator{\KG}{KG}

\DeclareMathOperator{\op}{op}

\newcommand{\ba}{\bar{\alpha}}

\newcommand\blfootnote[1]{%
  \begingroup
  \renewcommand\thefootnote{}\footnote{#1}%
  \addtocounter{footnote}{-1}%
  \endgroup
  }

\begin{document}

\title{On Krull-Gabriel dimension of weighted surface algebras}

\date{}
\author{K. Erdmann$^{a}$, A. Jaworska-Pastuszak$^{b,}$\footnote{Corresponding author }\;, G. Pastuszak$^{b}$}

\maketitle

{\footnotesize 
\noindent$^{a}$  Mathematical Institute, 
University of Oxford, ROQ, Oxford OX2 6GG, United Kingdom\\
$^{b}$ Faculty of Mathematics and Computer Science, Nicolaus
Copernicus University, Chopina 12/18, 87-100 Toru\'n, Poland}

\blfootnote{\noindent E-mail addresses:
    \noindent erdmann@maths.ox.ac.uk (K.~Erdmann), 
 jaworska@mat.umk.pl (A.~Jaworska-Pastuszak), past@mat.umk.pl (G.~Pastuszak)}

\begin{abstract}
We determine the Krull-Gabriel dimension of weighted surface algebras, a class of algebras which recently appeared in the context of classification of tame symmetric periodic algebras of non-polynomial growth. Moreover, we consider Krull-Gabriel dimension of idempotent algebras of weighted surface algebras and generalize the result in some cases.
\bigskip


\end{abstract}

\section{Introduction}

Let $K$ be an algebraically closed field. Throughout the paper, by an algebra we mean a basic indecomposable finite dimensional associative $K$-algebra with identity. For basic background on the representation theory of algebras we refer to \cites{ASS, SiSk3}, in particular, for the definitions of tame and wild representation type, as well as the stratification of tame representation type into domestic, polynomial and non-polynomial growth.

It is well known that any basic algebra is isomorphic to a bound quiver $K$-algebra. Moreover, we may view a bound quiver $K$-algebra as a finite locally bounded $K$-category \cite{BoGa}. Let $R$ be a locally bounded $K$-category. By $\mod(R)$ we denote the category of all finitely generated right $R$-modules and by $\mod(K)$ the category of all finite dimensional $K$-vector spaces. Let $\CF(R)$ be the category of all finitely presented contravariant  $K$-linear functors $\mod(R) \ra \mod(K)$, see \cite{Pr} for details on functor categories. This category, as an essentially small abelian $K$-category, admits the \emph{Krull-Gabriel filtration} \cite{Ge2}
$$\CF(R)_{-1}\subseteq\CF(R)_{0}\subseteq\CF(R)_{1}\subseteq\hdots\subseteq\CF(R)_{\alpha}\subseteq\CF(R)_{\alpha+1}\subseteq\hdots$$ 
by Serre subcategories, defined recursively as follows:
\begin{enumerate}[\rm(1)]
	\item $\CF(R)_{-1}=0$ and $\CF(R)_{\alpha+1}$ is the Serre subcategory of $\CF(R)$ formed by all functors having finite length in the quotient
category $\CF(R)\slash\CF(R)_{\alpha}$, for any ordinal number $\alpha$,
	\item $\CF(R)_{\beta}=\bigcup_{\alpha<\beta}\CF(R)_{\alpha}$, for any limit ordinal $\beta$.
\end{enumerate}
Following \cites{Ge1, Ge2}, the \emph{Krull-Gabriel dimension} $\KG(R)$ of $R$ is the smallest ordinal number $\alpha$ such that $\CF(R)_{\alpha}=\CF(R)$, if such
a number exists, and $\KG(R)=\infty$ otherwise. The Krull-Gabriel dimension of $R$ is \emph{finite} if $\KG(R)\in\NN$, and it is said to be \emph{undefined} if $\KG(R)=\infty$. 
Our interest in the Krull-Gabriel dimension is related with the following conjecture of Prest \cite{Pr}.
\begin{con}
Let $A$ be a finite dimensional algebra over an algebraically closed field $K$. Then $A$ is of domestic representation type if and only if  the Krull-Gabriel dimension $\KG(A)$ of $A$ is finite.
\end{con}

All known results support the conjecture. We just mention the fundamental one due to Auslander \cite{Au} which establishes an equivalence between the finite representation type and the Krull-Gabriel dimension zero. We refer, for example, to \cite{GP2}*{Section 1} for a comprehensive list of results in this direction. 

The aim of this paper is to determine the Krull-Gabriel dimension of \emph{weighted surface algebras} \cite{ES1}, and also of
some classes of \emph{hybrid algebras} \cite{ES6} and \emph{algebras of generalized quaternion type} \cite{ES3}. In particular, we confirm the conjecture of Prest for these classes of algebras. We shall now briefly describe the context in which these algebras appeared, as well as their significance.

Assume that $A$ is an algebra. Recall that for a module $M$ in  $\mod (A)$, its \emph{syzygy} is defined as the kernel $\Omega_A(M)$ of the minimal projective cover of $M$ in $\mod (A)$. Then $M$  is said to be \emph{periodic} if $\Omega_A^n(M) \cong M$ for some $n \geq 1$, and the minimal such $n$ is called the \emph{period} of $M$. An algebra $A$ is \emph{periodic} if it is periodic as a module over the enveloping algebra $A^{e}=A_{\op}\otimes_K A$, that is $\Omega^n_{A^e}(A)\cong A$, for some $n \geq 1$. It is known that if $A$ is a periodic algebra of period $n$, then for any indecomposable non-projective module $M\in \mod (A)$ we have $\Omega^n_{A}(M)\cong M$, i.e. $M$ is periodic in $\mod (A)$. Note that any periodic algebra is \emph{selfinjective} (that is, the projective modules in $\mod (A)$ are injective). A selfinjective algebra is called \emph{symmetric} if there exists an associative non-degenerate symmetric $K$-bilinear form $(-,-): A \times A \rightarrow K$.

The problem of classification, up to the Morita equivalence, of tame periodic algebras seems to be an ambitious task. Due to a result of Dugas \cite{Du}, it is known that all selfinjective algebras of finite type are periodic. Moreover, the representation-infinite tame periodic algebras of polynomial growth were completely classified by Bia{\l}kowski, Erdmann and Skowro\'nski in \cite{BES}. In particular, there is a common bound of periods of algebras in this case. In turn, for the non-polynomial case much less is known. Some symmetric periodic algebras of non-polynomial growth appeared naturally in the study of blocks of group algebras with generalized quaternion defect groups. This gave rise to the research on \emph{algebras of quaternion type}, a class of symmetric algebras introduced by Erdmann, see \cite{E} for more details and classification. It was shown in \cites{ES0q, Ho} that every algebra of quaternion type is a tame periodic algebra of period $4$. 

The class of \emph{weighted surface algebras} was introduced by Erdmann and Skowro\'nski in \cite{ES1}, inspired by some results in cluster theory. These algebras are symmetric algebras defined in terms of surface triangulations. The class was soon after extended to the general form given in \cite{ES4}. In this way, almost all algebras of quaternion type become weighted surface algebras. It follows from \cites{ES4,ES5} that weighted surface algebras are tame of non-polynomial growth periodic algebras of period $4$, except a few cases. Further, it was shown in \cite{ES3} that they play an important role in the characterization of \emph{algebras of generalized quaternion type}, that is, symmetric representation-infinite tame algebras for which all simple modules are periodic of period $4$. Another direction of developing the theory was recently given in \cite{ES6}, where \emph{hybrid algebras} were introduced as a generalization of blocks of group algebras with semidihedral defect group. These algebras form a large class of symmetric tame or finite type algebras, which contains on one extreme Brauer graph algebras, and on the other weighted surface algebras. It was shown that hybrid algebras are precisely the blocks of idempotent algebras of weighted surface algebras, up to socle deformations \cite{ES6}, and hence they do not need to be periodic. Note that a different approach to the problem of classification of tame symmetric algebras of period $4$ was presented in \cite{EHS} where these algebras are studied in terms of combinatorial properties of their Gabriel quiver. 

The following theorem is the main result of the paper, see Theorem \ref{wsa}.
\begin{thmintro}
The Krull-Gabriel dimension of a weighted surface algebra is undefined.
\end{thmintro}

Moreover, we show analogous facts about Krull-Gabriel dimension for some classes of hybrid algebras and algebras of generalized quaternion type, see Proposition \ref{hybrid} and Corollary \ref{quaternion}, respectively. We conclude in Theorem \ref{support} that the class of weighted surface algebras supports the conjecture of Prest. In Corollary \ref{soc} we obtain some natural generalisation of Theorem \ref{wsa}.

Proofs of our results are obtained, inter alia, by applying recent results of Pastuszak, relating Krull-Gabriel dimension and covering techniques. Namely, it was shown in \cites{GP0,GP1}  that if $R$ is a \emph{locally support-finite} \cite{DoSk} and \emph{intervally finite} \cite{BoGa} locally bounded $K$-category and $G$ is a torsion-free admissible group of $K$-linear automorphisms of $R$, then $\KG(R)=\KG(R\slash G)$ where $R\slash G$ denotes the orbit category of $R$ \cite{BoGa}. In other words, the induced Galois covering $R\ra R\slash G$ preserves the Krull-Gabriel dimension.
This theorem appeared to be very useful. Indeed, it was applied  in \cite{GP0}*{Theorem 7.3, Theorem 8.1} to determine Krull-Gabriel dimensions of tame locally support-finite repetitive categories and standard selfinjective algebras of polynomial growth. Moreover, in \cite{J-PP} Jaworska-Pastuszak and Pastuszak applied the theorem to relate and describe Krull-Gabriel dimensions of repetitive categories, cluster repetitive categories and
cluster-tilted algebras. In particular, it was shown that Prest conjecture is valid also for cluster-tilted algebras. A new, more general result shows that Galois coverings do not increase Krull-Gabriel dimension, that is, for a Galois covering $R \rightarrow A$ we always have $\KG(R)\leq \KG(A)$ \cite{GP2}, see also \cite{P7}. This is particularly useful in determining Krull-Gabriel dimension of $A$, if one is able to show that $\KG(R)=\infty$, since then $\KG(A)=\infty$. Our main results are based on this pattern of thinking.

The paper is organized as follows. Section 2 is devoted to present some facts about Galois coverings, Krull-Gabriel dimension and the way these concepts are related. This is based mostly on the results of \cite{GP0,GP2}, see also \cite{P7}. In Section 3 we give an insight into the nature of algebras which we are interested in, especially of weighted surface algebras (in the general version) and hybrid algebras. To this end, we review from \cite{ES5}, \cite{ES6} as much as needed. We believe that this section provides a convenient introduction to the theory. In Section 4 we prove our main results. We start with Proposition \ref{hybrid} where we determine the Krull-Gabriel dimension of some particular class of hybrid algebras. Afterwards we give examples showing the necessity of assumptions in this theorem, see Examples \ref{ex1} and \ref{ex2}. Then we apply \ref{hybrid} in the proof of the main Theorem \ref{wsa} which states that weighted surface algebras have Krull-Gabriel dimension undefined. In Theorem \ref{support} we show that the class of weighted surface algebras supports the conjecture of Prest. Finally, we draw some conclusions. Firstly, we present in Corollary \ref{soc} some natural generalisation of Theorem \ref{wsa}. Then in Corollary \ref{quaternion} we determine the Krull-Gabriel dimension of some class of algebras of generalized quaternion type. In Remarks \ref{rem_fin1} and \ref{rem_fin2} we make some comments concerning \emph{higher tetrahedral algebras} \cite{ES2} and \emph{higher spherical algebras} \cite{ES7}. 

\section{Galois coverings and Krull-Gabriel dimension}

We start this section with some general results on Krull-Gabriel dimension. The following two facts will be useful in our considerations. For the proofs we refer the reader to \cite{GP2}*{Lemma 2.4 and Lemma 2.6}, respectively, see also \cite{Kr}*{Appendix B}.

\begin{lm} \label{0}
Assume that $\mathcal{C}$, $\mathcal{D}$ are abelian categories and $F: \mathcal{C} \rightarrow \mathcal{D}$ is an exact functor.
\begin{itemize}
  \item[$(a)$] If $F$ is full and dense, then $\KG(\mathcal{D})\leq \KG(\mathcal{C})$.
  \item[$(b)$] If $F$ is faithful, then $\KG(\mathcal{C})\leq \KG(\mathcal{D})$.
\end{itemize}
\end{lm}

\begin{lm} \label{00}
Let $C$ and $D$ be locally bounded $K$-categories. 
\begin{itemize}
\item[$(a)$] If $D$ is a finite convex subcategory of $C$, then $\KG(D)\leq \KG(C)$.
\item[$(b)$] If $D$ is a quotient category of $C$, then $\KG(D)\leq \KG(C)$.
\end{itemize}
\end{lm}

Since a finite locally bounded $K$-category is identified with a bound quiver $K$-algebra, the inequalities in Lemma \ref{00} are also valid for subalgebras and quotient algebras of finite dimensional $K$-algebras.

Assume that $R,A$ are locally bounded $K$-categories, $F:R\ra A$ is a $K$-linear functor and $G$ a group of $K$-linear automorphisms of $R$ acting freely on the objects of $R$ (i.e. $gx=x$ if and only if $g=1$, for any $g\in G$ and $x\in\ob(R)$). Following \cite{BoGa}, we call $F:R\ra A$ a \emph{Galois covering} if and only if 
\begin{enumerate}[\rm(i)]
	\item the functor $F:R\ra A$ induces isomorphisms $$\bigoplus_{g\in G}R(gx,y)\cong A(F(x),F(y))\cong\bigoplus_{g\in G}R(x,gy)$$ of vector spaces, for any $x,y\in\ob(R)$,
	\item the functor $F:R\ra A$ induces a surjective function $\ob(R)\ra\ob(A)$,
	\item $Fg=F$, for any $g\in G$,
	\item for any $x,y\in\ob(R)$ such that $F(x)=F(y)$ there is $g\in G$ such that $gx=y$. 
\end{enumerate}
It is well known that $F:R\ra A$ satisfies the above conditions if and only if $F$ induces an isomorphism $A\cong R\slash G$ where $R\slash G$ is the \emph{orbit category} of $R$, see \cite{BoGa}.

The following theorem, stating that a Galois covering does not increase Krull-Gabriel dimension, was proved in \cite{GP2}*{Theorem 3.8}. This result is crucial in the proofs from Section 4.
\begin{thm} \label{A}
Assume that $F:R\ra A$ is a Galois covering of locally bounded $K$-categories $R$ and $A$. Then $\KG(R)\leq\KG(A)$ and in particular:
\begin{enumerate}[\rm(1)]
  \item If $\KG(R)$ is undefined, then $\KG(A)$ is undefined.
  \item If $\KG(A)$ is finite, then $\KG(R)$ is finite.
\end{enumerate}
\end{thm}

A standard example of a Galois covering functor is the Galois covering $F: \widehat{B} \rightarrow \T(B)$ of the trivial extension $\T(B)$ of an algebra $B$ by the repetitive category $\widehat{B}$ of $B$.
By general theory (see \cite{Sk2}), the trivial extension algebra $\T(B)$ is isomorphic to the orbit algebra $\widehat{B}/(\nu_{\widehat B})$ of the category $\widehat{B}$ with respect to the infinite cyclic group $(\nu_{\widehat B})$ of automorphisms of $\widehat{B}$ generated by the \emph{Nakayama automorphism} $\nu_{\widehat B}$. Therefore we have the immediate corollary of Theorem \ref{A}.
\begin{cor} \label{000}
For any finite dimensional $K$-algebra $B$ we have $\KG(\widehat{B}) \leq \KG(\T(B))$.
\end{cor}
Recall that the group $G$ of automorphisms of a locally bounded $K$-category $R$ is called \emph{admissible} if and only if $G$ acts freely on the objects of $R$ and there are only finitely many $G$-orbits. A category $R$ is \emph{locally support-finite} \cite{DoSk} if and only if for any object $x$ of $R$ the union of supports of all indecomposable modules $M$ such that $M(x)\neq 0$, is finite. A category $R$ is \emph{intervally finite} \cite{BoGa}, if any finite subcategory of $R$ has a finite convex hull. We have the following fact proved in \cite{GP0}*{Theorem 7.3}.

\begin{thm}\label{B}
Assume that $R$ is a locally support-finite and intervally finite locally bounded $K$-category and $G$ is an admissible torsion-free group of $K$-linear automorphisms of $R$. Assume that $A=R\slash G$ is the orbit category and $F: R \ra A$ the associated Galois covering. Then $\KG(R)=\KG(A)$.
\end{thm}

In the study of weighted surface algebras and hybrid algebras, the following classes of algebras play a prominent role: \emph{string algebras} \cite{BuRi,SkWa}, \emph{pg-critical algebras} \cite{NorSko} and \emph{standard selfinjective algebras} \cite{Sk2}. The following theorem describes Krull-Gabriel dimensions of these classes. We apply these results in the forthcoming sections.



\begin{thm} \label{00000}
Let $A$ be a finite dimensional algebra over an algebraically closed field $K$. The following assertions hold.
\begin{enumerate}[\rm(1)] 
\item \label{Sch} If $A$ is a non-domestic string algebra, then $\KG(A)=\infty$.
\item \label{KP} If $A$ is a pg-critical algebra, then $\KG(A)=\infty$.
\item \label{P} If $A$ is standard selfinjective algebra of infinite representation type, then:
\begin{enumerate}[\rm(a)]
\item If $A$ is domestic, then  $\KG(A)=2$,
\item If $A$ non-domestic of polynomial growth, then $\KG(A)=\infty$.
\end{enumerate}
\end{enumerate}
\end{thm}

\begin{proof} The proof of (\ref{Sch}) was given by Schr\"{o}er in \cite{Sch}*{Proposition 2}. The implication of (\ref{KP}) was proved by Kasjan and Pastuszak in \cite{KP}. In fact, the results of \cite{KP} show that the width of the lattice of all pointed modules over pg-critical algebras is infinite. This yields that Krull-Gabriel dimension is undefined, see for example \cite{Pr}. The assertion of (\ref{P}) was proved in \cite{GP0}*{Theorem 8.1}, based on Theorem \ref{B}.
\end{proof}


\section{Weighted surface algebras and hybrid algebras}
In this section we recall the definitions and basic properties of weighted surface algebras and, more generally, of hybrid algebras.
Definitions are presented  at the technical level that is sufficient for our purposes, for details we refer to \cite{ES1,ES4,ES5,ES6}.

Recall that for a quiver $Q$, by $Q_0$ and $Q_1$ we denote the set of vertices and the set of arrows between them, respectively. By $KQ$ we denote the path algebra of $Q$ over $K$, see \cite{ASS}. 
The underlying space has as basis the set of all paths in $Q$.
Let $R_Q$ be the ideal of $KQ$ generated by all paths of length $\geq 1$.
We will consider algebras of the form $A=KQ/I$ where $I$ is an ideal of $KQ$
which contains $R_Q^m$ for some $m\geq 2$, so that
the algebra is finite-dimensional and basic. Note that $I$ does not need to be admissible, that is, $I$ may not be contained in $R_Q^2$.
Therefore the Gabriel quiver $Q_A$ of $A$ is the full subquiver of $Q$
obtained from $Q$ by removing all arrows $\alpha$ with $\alpha +I \in R_Q^2+I$.

We say that a quiver $Q$ is \emph{$2$-regular} if for each vertex $i\in Q_0$ there are precisely two arrows starting at $i$ and two arrows ending at $i$. Any $2$-regular quiver admits an involution on arrows, $\alpha \mapsto \ba$, such that $\ba$ is different from $\alpha$ but starts at the same vertex. Further, a \emph{biserial quiver} is a pair $(Q, f)$ where $Q$ is a (finite) connected 2-regular quiver and  $f$ is a fixed permutation of the arrows such that $f(\alpha)$ starts at the vertex that $\alpha$ ends at. 
The set of arrows in the $f$-orbit of length $3$ or $1$ is called a \emph{triangle}. If all cycles of $f$ form triangles, that is $f^3=\id$, then we say that $(Q, f)$ is a \emph{triangulation quiver}. Note that in any biserial quiver the permutation $f$ uniquely determines a permutation $g$ on the set of arrows, given by $g(\alpha):=\overline{f(\alpha)}$ for any arrow $\alpha$. Further, on the set of $g$-orbits $\mathcal{O}(g)$ we define $m_{\bullet}: \mathcal{O}(g) \ra \mathbb{N}^*$ and $c_{\bullet}: \mathcal{O}(g) \ra K^*$, called \emph{the weight} and \emph{the parameter} function, respectively. We set $m_\alpha=m_{\bullet}(\mathcal{O}(\alpha))$ and $c_\alpha=c_{\bullet}(\mathcal{O}(\alpha))$, for any arrow $\alpha$. By $n_{\alpha}$ we denote the length of the $g$-orbit of $\alpha$ and if $m_{\alpha}n_{\alpha}=2$ we call $\alpha$ to be a \emph{virtual} arrow.

Let now $T$ be a triangulation of a compact connected real surface $S$ with an orientation $\vec{T}$ of triangles. Then with a directed triangulated surface $(S,\vec{T})$ we associate a triangulation quiver ($Q(S, \vec{T}),f)$, where a permutation $f$ of arrows in $Q(S, \vec{T})$ reflects the orientation $\vec{T}$ of triangles in $T$, see e.g. \cite{ES1}. Note that, conversely, every triangulation quiver $(Q,f)$ is a triangulation quiver $Q(S, \vec{T})$ associated to some directed triangulated surface $(S, \vec{T})$  (see \cite{ES1}*{Theorem 4.11} and also \cite{L}). Therefore, in the forthcoming, we will not distinguish between these quivers.

The algebra $$\La = \La(Q, f, m_{\bullet}, c_{\bullet}) = KQ/I$$
is called  a \emph{weighted surface algebra} if
$(Q, f)=(Q(S, \vec{T}),f)$ is a triangulation quiver (of the surface $(S,\vec{T})$) with $|Q_0| \geq 2$, $m_{\alpha}n_{\alpha}\geq 2$ for any arrow $\alpha$,
and $I= I(Q, f, m_{\bullet}, c_{\bullet})$ is the ideal of
$KQ$ generated by:
\begin{enumerate}[(i)]
 \item
	$\alpha f(\alpha) - c_{\ba}A_{\ba}$ for all arrows $\alpha$ of $Q$, where $A_{\alpha}:= \alpha g(\alpha)\ldots g^{m_{\alpha}n_{\alpha}-2}(\alpha)$ is  the path along $g$-orbit of length $m_{\alpha}n_{\alpha}-1$,
 \item
	$\alpha f(\alpha) g(f(\alpha))$ \ for all arrows $\alpha$ of $Q$
	unless $\ba$ is virtual, or unless $f(\ba)$ is virtual and $m_{\ba}=1$, $n_{\ba}=3$,
 \item
	$\alpha g(\alpha)f(g(\alpha))$ for all arrows $\alpha$ of $Q$ unless $f(\alpha)$ is virtual, or unless $f^2(\alpha)$ is virtual and $m_{f(\alpha)}=1$, $n_{f(\alpha)}=3$.
\end{enumerate}

The above definition comes from \cite{ES4} where the authors generalized the definition of weighted surface algebras given in \cite{ES1} (see \cite{ES4}*{Section 2} for details). The new approach does not require the assumption that the Gabriel quiver is $2$-regular. This implies, for example, that almost all algebras of quaternion type \cite{E} become weighted surface algebras in the sense of \cite{ES4}.

\begin{rmk} The ideal $I$ of a weighted surface algebra $\Lambda=KQ/I$ does not have to be admissible. Indeed, observe that if $\alpha$ is a virtual arrow, then by $\rm{(i)}$ we have $\ba f(\ba)-c_{\alpha}\alpha \in I$ and hence $I\nsubseteq R_Q^2$. We conclude that the Gabriel quiver $Q_{\La}$ of $\La$ is obtained from $Q$ by removing all virtual arrows.
\end{rmk}



Among the class of weighted surface algebras, some families of algebras are distinguished due to their exceptional properties. These are the \emph{disc algebras} $D(\lambda)$,  the \emph{tetrahedral algebras} $\Lambda(\lambda)$, the \emph{triangle algebras} $T(\lambda)$ and the \emph{spherical algebras} $S(\lambda)$, where $\lambda \in K^*$. The names of the families are related to the triangulations of surface, namely the unit disc $D^2$ in $\mathbb{R}^2$ in the case of $D(\lambda)$ and the sphere $S^2$ in $\mathbb{R}^3$ for other cases. For any algebra $A$ from these families, by $(Q(A),f)$ we denote the underlying triangulation quiver of $A$. For the purpose of this article we recall now the quotients $A=KQ(A)/I$, referring the reader to \cite{ES4}*{Section 3} for more details on these algebras.

\begin{itemize}
 \item[(1)] A \emph{disc algebra} $D(\lambda)$, for $\lambda \in K^*$, is given by the $2$-regular quiver $Q(D(\lambda))$ of the form:
\[
  \xymatrix{
    1
    \ar@(ld,ul)^{\alpha}[]
    \ar@<.5ex>[r]^{\beta}
    & 2
    \ar@<.5ex>[l]^{\gamma}
    \ar@(ru,dr)^{\sigma}[]
  }
,
\]
the permutation $f$ with orbits $(\alpha\; \beta\; \gamma)$, $(\sigma)$
and the following relations:
\begin{align*}
 \beta \gamma &= \lambda \alpha^2 ,
 &
 \alpha\beta &=  \beta \sigma ,
 &
 \gamma \alpha &=  \sigma \gamma ,
 &
 \sigma^2 &=  \gamma \beta ,
 &
 \alpha\beta\sigma &= 0,
 &
 \beta\gamma\beta &= 0 ,
\\
 \gamma \alpha^2 &= 0,
 &
 \sigma^2 \gamma &= 0 ,
 &
 \alpha^2 \beta &= 0,
 &
 \beta \sigma^2 &= 0 ,
 &
 \sigma \gamma \alpha &= 0,
 &
 \gamma \beta \gamma &= 0 .
\end{align*}
The weight function is given by $m_{\alpha}=3$ and $m_{\beta}=1$, and the parameter function is  $c_{\alpha}=\lambda$ and $c_{\beta}=1$.

\item[(2)] A \emph{tetrahedral algebra} $\Lambda(\lambda)$, for $\lambda \in K^*$, is given by the $2$-regular quiver $Q(\Lambda(\lambda))$:
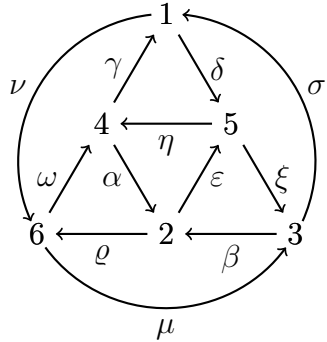
\[
\begin{tikzpicture}
[scale=.85]
\node (1) at (0,1.72) {$1$};
\node (2) at (0,-1.72) {$2$};
\node (3) at (2,-1.72) {$3$};
\node (4) at (-1,0) {$4$};
\node (5) at (1,0) {$5$};
\node (6) at (-2,-1.72) {$6$};
\coordinate (1) at (0,1.72);
\coordinate (2) at (0,-1.72);
\coordinate (3) at (2,-1.72);
\coordinate (4) at (-1,0);
\coordinate (5) at (1,0);
\coordinate (6) at (-2,-1.72);
\node (1) at (0,1.72) {$1$};
\node (2) at (0,-1.72) {$2$};
\node (3) at (2,-1.72) {$3$};
\node (4) at (-1,0) {$4$};
\node (5) at (1,0) {$5$};
\node (6) at (-2,-1.72) {$6$};
\draw[->,thick] (-.23,1.7) arc [start angle=96, delta angle=108, radius=2.3cm] node[midway,left] {$\nu$};
\draw[->,thick] (-1.87,-1.93) arc [start angle=-144, delta angle=108, radius=2.3cm] node[midway,below] {$\mu$};
\draw[->,thick] (2.11,-1.52) arc [start angle=-24, delta angle=108, radius=2.3cm] node[midway,right] {$\sigma$};
\draw[->,thick]
 (1) edge node [right] {$\delta$} (5)
(2) edge node [right] {$\varepsilon$} (5)
(2) edge node [below] {$\varrho$} (6)
(3) edge node [below] {$\beta$} (2)
 (4) edge node [left] {$\gamma$} (1)
(4) edge node [left] {$\alpha$} (2)
(5) edge node [right] {$\xi$} (3)
 (5) edge node [below] {$\eta$} (4)
(6) edge node [left] {$\omega$} (4)
;
\end{tikzpicture}
\]
the permutation $f$ with orbits $(\delta\; \eta\; \gamma)$, $(\alpha\; \varrho\; \omega)$, $(\xi\; \beta\; \varepsilon)$, $(\nu\; \mu\; \sigma)$
and  the following relations:
\begin{align*}
  \gamma \delta &= \lambda \alpha \varepsilon ,
  &
  \varrho \omega &= \lambda \varepsilon \eta ,
  &
  \xi \beta &= \lambda \eta \alpha ,
  &
  \alpha \varrho &= \gamma \nu ,
  &
 \delta \eta &= \nu \omega ,
  &
  \omega \alpha &= \mu \beta ,
  \\
  \beta \varepsilon &= \sigma \delta ,
  &
  \varepsilon \xi &= \varrho \mu ,
  &
  \eta \gamma &= \xi \sigma ,
  &
  \sigma \nu &= \beta \varrho ,
  &
  \nu \mu &= \delta \xi ,
  &
  \mu \sigma &= \omega \gamma ,
\end{align*}
\begin{center}
$\tau f(\tau)g(f(\tau)) =0 \quad \text{and} \quad   \tau g(\tau)f(g(\tau))=0 \qquad$ for all arrows $\tau \in Q_1.$ \\
\end{center}
Here the weight function is trivial, and $c_{\alpha}=\lambda$ and the other parameters equal $1$.
\end{itemize}
Note that there is a connection between the disc algebra $D(\lambda)$ and the tetrahedral algebra $\Lambda(\lambda)$, for any $\lambda \in
K^*$. Namely, the  cyclic group of order $3$ acts on  $\Lambda(\lambda)$  by  cyclic rotation of vertices and arrows of the quiver
$Q(\Lambda(\lambda))$:
\begin{align*}
 (1\ 6\ 3),
  &&
 (4\ 2\ 5),
  &&
 (\alpha\ \varepsilon\ \eta),
  &&
  (\beta\ \delta\ \omega),
  &&
  (\gamma\ \varrho\ \xi),
  &&
  (\sigma\ \nu\ \mu) .
\end{align*}
Then $D(\lambda)$ is the orbit algebra $\Lambda(\lambda)/\mathbb{Z}_3$.\\
\begin{itemize}
\item[(3)] A \emph{triangle algebra} $T(\lambda)$, for $\lambda \in K^*$, is given by the $2$-regular quiver $Q(T(\lambda))$:
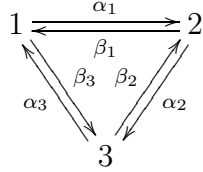
\[
  \xymatrix@R=3.pc@C=1.8pc{
    1
    \ar@<.35ex>[rr]^{\alpha_1}
    \ar@<.35ex>[rd]^{\beta_3}
    && 2
   \ar@<.35ex>[ll]^{\beta_1}
    \ar@<.35ex>[ld]^{\alpha_2}
    \\
    & 3
    \ar@<.35ex>[lu]^{\alpha_3}
    \ar@<.35ex>[ru]^{\beta_2}
  }
\]
the permutation $f$ with orbits
$(\alpha_1\; \alpha_2\; \alpha_3)$ and $(\beta_1\; \beta_3\; \beta_2)$ and the
relations:
\begin{align*}
 \lambda \alpha_1 \beta_1 \alpha_1 &= \beta_3 \beta_2 ,
 &
\lambda \beta_1 \alpha_1 \beta_1 &= \alpha_2 \alpha_3,
 &
\beta_2 \alpha_2 \beta_2 &= \alpha_3 \alpha_1,
& 
\alpha_2 \beta_2 \alpha_2&= \beta_1 \beta_3 ,
\\
 &&
 \beta_2 \beta_1 &=  \alpha_3 ,
 &
 \alpha_1 \alpha_2 &=  \beta_3 ,
 &&
\\
 \alpha_2 \alpha_3 \beta_3 &= 0,
 &
 \alpha_3 \alpha_1 \beta_1 &= 0,
 &
 \beta_1 \beta_3 \alpha_3 &= 0,
 &
 \beta_3 \beta_2 \alpha_2 &= 0,
\\
\alpha_1 \beta_1 \beta_3 &= 0,
 &
 \alpha_3 \beta_3 \beta_2 &= 0,
 &
 \beta_2 \alpha_2 \alpha_3 &= 0,
 &
 \beta_3 \alpha_3 \alpha_1 &= 0.
\end{align*}
The weight function is given by $m_{\alpha_1}=2=m_{\alpha_2}$ and $m_{\alpha_3}=1$, in particular $\alpha_3$ and $\beta_3$ are virtual. The parameters are $c_{\alpha_1} =\lambda$ and the others are equal to $1$.

\item[(4)] A \emph{spherical algebra} $S(\lambda)$, for $\lambda \in K^*$, is given by the $2$-regular quiver
$Q(S(\lambda))$:
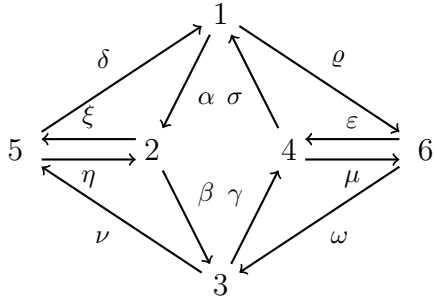
\[
\begin{tikzpicture}
[->,scale=.9]
\coordinate (1) at (0,2);
\coordinate (2) at (-1,0);
\coordinate (2u) at (-.925,.15);
\coordinate (2d) at (-.925,-.15);
\coordinate (3) at (0,-2);
\coordinate (4) at (1,0);
\coordinate (4u) at (.925,.15);
\coordinate (4d) at (.925,-.15);
\coordinate (5) at (-3,0);
\coordinate (5u) at (-2.775,.15);
\coordinate (5d) at (-2.775,-.15);
\coordinate (6) at (3,0);
\coordinate (6u) at (2.775,.15);
\coordinate (6d) at (2.775,-.15);
\node [fill=white,circle,minimum size=4.5] (1) at (0,2) {\ \quad};
\node [fill=white,circle,minimum size=4.5] (2) at (-1,0) {\ \quad};
\node [fill=white,circle,minimum size=4.5] (3) at (0,-2) {\ \quad};
\node [fill=white,circle,minimum size=4.5] (4) at (1,0) {\ \quad};
\node [fill=white,circle,minimum size=4.5] (5) at (-3,0) {\ \quad};
\node [fill=white,circle,minimum size=4.5] (6) at (3,0) {\ \quad};
\node (1) at (0,2) {1};
\node (2) at (-1,0) {2};
\node (2u) at (-1,0.15) {\ \quad};
\node (2d) at (-1,-0.15) {\ \quad};
\node (3) at (0,-2) {3};
\node (4) at (1,0) {4};
\node (4u) at (1,0.15) {\ \quad};
\node (4d) at (1,-0.15) {\ \quad};
\node (5) at (-3,0) {5};
\node (5u) at (-2.775,0.15) {};
\node (5d) at (-2.775,-0.15) {};
\node (6) at (3,0) {6};
\node (6u) at (2.775,0.15) {};
\node (6d) at (2.775,-0.15) {};
\draw[thick,->]
(1) edge node[below right]{\footnotesize$\alpha$} (2)
(2u) edge node[above]{\footnotesize$\xi$} (5u)
(5u) edge node[above left]{\footnotesize$\delta$} (1)
(5d) edge node[below]{\footnotesize$\eta$} (2d)
(2) edge node[above right]{\footnotesize$\beta$} (3)
(3) edge node[below left]{\footnotesize$\nu$} (5d)
(1) edge node[above right]{\footnotesize$\varrho$} (6u)
(6u) edge node[above]{\footnotesize$\varepsilon$} (4u)
(4) edge node[below left]{\footnotesize$\sigma$} (1)
(4d) edge node[below]{\footnotesize$\mu$} (6d)
(6d) edge node[below right]{\footnotesize$\omega$} (3)
(3) edge node[above left]{\footnotesize$\gamma$} (4)
;
\end{tikzpicture}
\]
the permutation $f$ with orbits $(\alpha\; \xi\; \delta)$, $(\varrho\; \varepsilon\; \sigma)$, $(\omega\; \gamma\; \mu)$, $(\beta\; \nu\; \eta)$
and the relations:
\begin{align*}
  \mu \omega &= \lambda \sigma \alpha \beta ,
  &
  \xi \delta &= \lambda \beta \gamma \sigma ,
  &
  \varrho \varepsilon &= \lambda \alpha \beta \gamma ,
  &
  \nu \eta &= \lambda \gamma \sigma \alpha ,
\\
 \sigma \varrho &=  \mu ,
 &
 \delta \alpha &=  \eta ,
 &
 \beta \nu &=  \xi ,
 &
 \omega \gamma &=  \varepsilon ,
\\
  \eta \beta &=   \delta \varrho \omega ,
  &
  \gamma \mu &=   \nu \delta \varrho ,
  &
  \alpha \xi &=  \varrho \omega \nu ,
  &
  \varepsilon \sigma &=  \omega \nu \delta ,
\\
  \alpha \xi \eta &= 0 ,
  &
  \xi \delta \varrho &= 0 ,
  &
  \nu \eta \xi &= 0 ,
  &
  \eta \beta \gamma &= 0 ,
\\
  \gamma \mu \varepsilon &= 0 ,
  &
  \mu \omega \nu &= 0 ,
  &
  \varrho \varepsilon \mu &= 0 ,
  &
  \varepsilon \sigma \alpha &= 0 ,
\\
  \beta \gamma \mu &= 0 ,
  &
  \sigma \alpha \xi &= 0 ,
  &
  \delta \varrho \varepsilon &= 0 ,
  &
  \omega \nu \eta &= 0 ,
\\
  \xi \eta \beta &= 0 ,
  &
  \eta \xi \delta &= 0 ,
  &
  \mu \varepsilon \sigma &= 0 ,
  &
  \varepsilon \mu \omega &= 0 .
\end{align*}
The weight function is trivial, and $\mu, \eta, \xi$ and $\varepsilon$
are virtual. The parameters are $c_{\sigma} = \lambda$ and the other parameters
are equal to $1$.
\end{itemize}
Following \cite[Section 3]{ES4} we conclude that there is a natural action of the cyclic group of order $2$ on $S(\lambda)$ given by the cyclic
rotation of vertices and arrows of the quiver $Q(S(\lambda))$:
\begin{align*}
 (1\ 3),
  &&
 (2\ 4),
  &&
 (5\ 6),
  &&
 (\alpha\ \gamma),
  &&
  (\beta\ \sigma),
  &&
  (\varrho\ \nu),
&&
  (\omega\ \delta) .
\end{align*}
Then the triangle algebra $T(\lambda)$ is the orbit algebra $S(\lambda)/\mathbb{Z}_2$, for any  $\lambda \in K^*$.

If $\lambda = 1$, then we call the algebras $D(1)$, $\Lambda(1)$, $T(1)$ and $S(1)$ the \emph{singular} disc, tetrahedral, triangle and spherical
algebra, respectively. If $\lambda \in K \setminus \{ 0,1\}$ we say that $D(\lambda)$, $\Lambda(\lambda)$, $T(\lambda)$ and $S(\lambda)$ are
\emph{non-singular}. The following facts from \cite{ES4} explain why disc, tetrahedral, triangle and spherical algebras should be considered
separately.

\begin{thm} \label{C}
Let $\Lambda$ be a weighted surface algebra. Then $\Lambda$ is tame of non-domestic type and the following assertions hold:
\begin{enumerate}[\rm(1)]
\item $\Lambda$ is periodic of period $4$ if and only if it is not isomorphic to $D(1), \Lambda(1), T(1), S(1)$.
\item $\Lambda$ is symmetric if and only if it is not isomorphic to $T(1)$, $S(1)$.
\item $\Lambda$ is tame of polynomial growth if and only if it is isomorphic to $D(\lambda)$, $\Lambda(\lambda)$, $T(\lambda)$ and $S(\lambda)$, $\lambda \in K\backslash \{0,1\}$.
\end{enumerate}
\end{thm}

\noindent In order to proof the tameness, Erdmann and Skowro\'nski showed that almost all weighted surface algebras degenerate to \emph{biserial weighted surface algebras} \cite{ES4}, yet another class of algebras associated to triangulated surfaces. Biserial weighted surface algebras are special biserial \cite[Proposition 2.11]{ES4} and thus tame by \cite{WW}. Then we conclude from \cite{Ge} that weighted surface algebras are tame as well. The fact that the algebras given in Theorem \ref{C}(3) are non-domestic of polynomial growth follows from \cite{Sk1}*{Theorem} and \cite{BiS}*{Theorem 4.2}, we refer also to the proof of Theorem \ref{support} for more details. To prove that the remaining weighted surface algebras are of non-polynomial growth, two more families of algebras were distinguished: $D(\lambda)^{(1)}$ and $D(\lambda)^{(2)}$. These families are related to different triangulations of the unit disc than in the case of the disc algebra $D(\lambda)$, see \cite[Example 6.8]{ES4}.

An algebra $D(\lambda)^{(1)}$, for $\lambda \in K^*$, is given by the $2$-regular quiver $Q(D(\lambda)^{(1)})$:
\[
  \xymatrix{
    1
    \ar@(ld,ul)^{\xi}[]
    \ar@<.5ex>[r]^{\beta}
    & 3
    \ar@<.5ex>[l]^{\alpha}
    \ar@(ru,dr)^{\gamma}[]
  }
\]
the permutation $f$ with orbits
$(\xi\ \beta\ \alpha)$, $(\gamma)$,
and the relations:
\begin{align*}
&&
 \alpha \xi &= \lambda \gamma \alpha \beta \gamma \alpha ,
 &
 \xi \beta &= \lambda \beta \gamma \alpha \beta \gamma ,
 &
 \gamma^2 &= \lambda \alpha \beta \gamma \alpha \beta ,
 &
 \beta \alpha &=  \xi ,
  &
 \\
 \alpha \xi^2 &= 0 ,
 &
 \xi \beta \gamma &= 0 ,
 &
 \gamma^2 \alpha &= 0 ,
 &
 \xi^2 \beta &= 0 ,
 &
 \gamma \alpha \xi &= 0 ,
 &
 \beta \gamma^2 &= 0.
\end{align*}

\noindent An algebra $D(\lambda)^{(2)}$, for $\lambda \in K^*$, is given by the $2$-regular quiver $Q(D(\lambda)^{(2)})$:
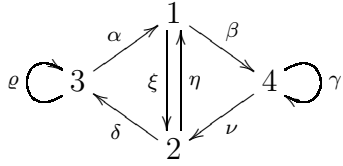
\[
  \xymatrix@R=1pc{
   & 1 \ar[rd]^{\beta} \ar@<-.5ex>[dd]_{\xi}  \\
   3   \ar@(ld,ul)^{\varrho}[]  \ar[ru]^{\alpha}
   && 4   \ar@(ru,dr)^{\gamma}[]  \ar[ld]^{\nu}   \\
   & 2 \ar[lu]^{\delta}  \ar@<-.5ex>[uu]_{\eta}
  }
\]
the permutation $f$ with orbits
$(\alpha\ \xi\ \delta)$, $(\beta\ \nu\ \eta)$, $(\varrho)$, $(\gamma)$
and the set of relations:
\begin{align*}
 \alpha \xi &= \lambda \varrho \alpha \beta \gamma \nu ,
 &
 \xi \delta &= \lambda \beta \gamma \nu \delta \varrho ,
 &
 \varrho^2 &= \lambda \alpha \beta \gamma \nu \delta ,
 &
 \delta \alpha &=  \eta ,
 &
 \alpha  \xi \eta &= 0 ,
 \\
 \nu \eta &= \lambda \gamma \nu \delta \varrho \alpha ,
 &
 \eta \beta &= \lambda \delta \varrho \alpha \beta \gamma ,
 &
 \gamma^2 &= \lambda \nu \delta \varrho \alpha \beta ,
 &
 \beta \nu &=  \xi ,
 &
 \xi \eta \beta &= 0 ,
 \\
 \xi \delta \varrho &= 0 ,
 &
 \nu \eta \xi &= 0 ,
 &
 \eta \beta \gamma &= 0 ,
 &
 \varrho^2 \alpha &= 0 ,
 &
  \gamma^2 \nu &= 0 ,
 \\
 \varrho \alpha \xi &= 0 ,
 &
 \eta \xi \delta &= 0 ,
 &
 \gamma \nu \eta &= 0 ,
 &
 \delta \varrho^2 &= 0 ,
 &
 \beta \gamma^2 &= 0 .
\end{align*}
Note that the quiver $Q(D(\lambda)^{(1)})$ is the orbit quiver $Q(D(\lambda)^{(2)})/H$ of the quiver $Q(D(\lambda)^{(2)})$
with respect to the action of the cyclic group $H$ of order $2$ given by the following cyclic rotation of vertices and arrows:
\begin{align*}
 &&
 (1\ 3) ,
 &&
 (2\ 4) ,
 &&
 (\alpha\ \nu) ,
 &&
 (\beta\ \delta) ,
 &&
 (\gamma\ \varrho) ,
 &&
 (\xi\ \eta).
 &&
\end{align*}

The following theorem was proved in \cite[Theorem 6.10]{ES4}.

\begin{thm}\label{D}
Let $\Lambda=\Lambda(Q, f, m_{\bullet}, c_{\bullet})$ be a weighted surface algebra. If $\Lambda$ is not isomorphic to
$\Lambda(\lambda)$, $T(\lambda)$, $S(\lambda)$, $D(\lambda)$, $D(\lambda)^{(1)}$, $D(\lambda)^{(2)}$, for any $\lambda \in K^*$, then there exists a quotient algebra $\Gamma$ of $\Lambda$ which is a string algebra of non-polynomial growth. In particular, $\Lambda$ is tame of non-polynomial growth.
\end{thm}

We recall from \cite[Section 6]{ES4} the construction of the quotient algebra $\Gamma$ from the above theorem. Namely, we have
$$\Gamma=\Gamma(Q,f,m_{\bullet},c_{\bullet}) = KQ/L(Q,f,m_{\bullet},c_{\bullet})$$ where $L(Q,f,m_{\bullet},c_{\bullet})$
is the ideal in the path algebra $K Q$ of $Q$ 
generated by the elements $\alpha f(\alpha)$ and $A_{\alpha}$,
for all arrows $\alpha \in Q_1$. Note that $\Gamma$ has a presentation $KQ_{\Gamma}/I_{\Gamma}$, where $Q_{\Gamma}$ is the Gabriel quiver of $\Gamma$ and $I_{\Gamma}=L(Q,f,m_{\bullet},c_{\bullet}) \cap KQ_{\Gamma}$ is the ideal generated by $\alpha f(\alpha)$ and $A_{\alpha}$ for non-virtual arrows $\alpha\in Q_1$. Then it is easy to see that $\Gamma$ is a string algebra, and this is in fact the largest string quotient
algebra of $\Lambda$ with respect to the dimension. One can check that in case of $D(\lambda)$, $\Lambda(\lambda)$, $T(\lambda)$, $S(\lambda)$, $D(\lambda)^{(1)}$ and $D(\lambda)^{(2)}$ that  such a quotient algebra is tame of domestic type, for any $\lambda\in K^*$. Nevertheless, it was shown using different arguments that $D(1)$, $\Lambda(1)$, $T(1)$, $S(1)$ as well as $D(\lambda)^{(1)}$, $D(\lambda)^{(2)}$ for any $\lambda \in K^*$, are of non-polynomial growth and hence Theorem \ref{C}(3) holds.

Recall that selfinjective algebras $A$ and $B$ are said to be \emph{socle equivalent} if the quotient
algebras $A/\soc(A)$ and $B/\soc(B)$ are isomorphic, where $\soc$ denotes the socle of an algebra. Socle equivalences of weighted surface algebras were studied in \cite{BEHSY}. In particular, it follows by \cite{BEHSY}*{Theorems 1.1 and 1.2} that if a symmetric algebra $A$ with Grothendieck group $K_0(A)$ of rank
at least $2$ is socle equivalent but not isomorphic to a weighted surface algebra $\Lambda(Q, f, m_{\bullet}, c_{\bullet})$, then the characteristic of the filed $K$ is $2$ and $(Q,f)$ is a triangulation quiver associated to a surface $S$ with non-empty boundary $\partial S$. The boundary $\partial S$ determines the \emph{border} $\partial(Q,f)$ of $(Q,f)$. This is the set of all vertices for which there is a loop $\alpha$ such that $f(\alpha)=\alpha$. These loops are called \emph{border loops}. Hence we can define a function $b_{\bullet}: \partial(Q,f) \rightarrow K$ called the \emph{border function}. Further, it follows by \cite{BEHSY}*{Theorem 1.2} that if  $A$ is socle equivalent but not isomorphic to the  weighted surface algebra $\Lambda(Q, f, m_{\bullet}, c_{\bullet})=KQ/I$, then $A$ is isomorphic to a \emph{socle deformed weighted surface algebra} $\Lambda(Q, f, m_{\bullet}, c_{\bullet}, b_{\bullet})=KQ/I'$, where $b_{\bullet}\neq 0$ and the ideal $I' = I'(Q, f,m_{\bullet}, c_{\bullet}, b_{\bullet})$ of the path algebra $KQ$ is generated by:
\begin{enumerate}
 \item[(i)]
	$\alpha f(\alpha) - c_{\ba}A_{\ba}$ for all arrows $\alpha$ of $Q$, which are not border loops,
 \item[(i')]
    $\alpha f(\alpha) - c_{\ba}A_{\ba}- b_{\ba}B_{\ba}$, for all border loops $\alpha$ in $Q$, where $b_{\alpha}$ is a value of $b_{\bullet}$ on the starting vertex of $\alpha$ and $B_{\alpha}:=A_{\alpha}g^{-1}(\alpha)=\alpha g(\alpha)\ldots g^{m_{\alpha}n_{\alpha}-1}(\alpha)$ is a path along $g$-orbit of length $m_{\alpha}n_{\alpha}$,
 \end{enumerate}
and relations of type (ii) and (iii) from the definition of weighted surface algebras. We refer to \cite[Section 2]{BEHSY} for more details. Let us only mention that paths $B_{\alpha}$ are generators of the socle of a weighted surface algebra $\Lambda(Q, f, m_{\bullet}, c_{\bullet})$. Observe that if a triangulation quiver $(Q,f)$ has non-empty border but the border function $b_{\bullet}=0$, then the socle deformed weighted surface algebra $\Lambda(Q, f, m_{\bullet}, c_{\bullet}, b_{\bullet})$ is by definition equal to weighted surface algebra $\Lambda(Q, f, m_{\bullet}, c_{\bullet})$. 


The following fact is used in the proof of Corollary \ref{soc}.

\begin{prop} \label{E}
Let $A$ be a symmetric algebra with Grothendieck group $K_0(A)$ of rank
at least $2$ which is socle equivalent but not isomorphic to a weighted surface algebra $\Lambda$. If $\Lambda$ is not isomorphic to any of algebras $D(\lambda)$, $D(\lambda)^{(1)}$, $D(\lambda)^{(2)}$, $\lambda \in K^*$, then there is a string quotient algebra $\Gamma$ of $A$ which is of non-polynomial growth.
\end{prop}
\begin{proof}
Assume that $A$ is socle equivalent but not isomorphic to a weighted surface algebra $\Lambda=\Lambda(Q, f, m_{\bullet}, c_{\bullet})$. 
By the above considerations, $\Lambda$ has non-empty border $\partial(Q,f)$ and $A\cong \Lambda(Q, f, m_{\bullet}, c_{\bullet}, b_{\bullet})$ for some non-zero border function $b_{\bullet}$ on $\partial(Q,f)$. 

Assume that $\Lambda$ is isomorphic to some algebra $\Lambda'$ from exceptional families $\Lambda(\lambda)$, $S(\lambda)$, $T(\lambda)$, $\lambda \in K^*$. Then $A$ is also socle equivalent to $\Lambda'$. Since $\Lambda'$ is associated to the surface without boundary (a sphere $S^2$ in $\mathbb{R}^3$), we get by \cite[Theorem 1.2]{BEHSY} that $A \cong \Lambda' \cong \Lambda$, a contradiction. Therefore $\Lambda$ is not isomorphic to any of the algebras $\Lambda(\lambda)$, $S(\lambda)$, $T(\lambda)$, for $\lambda \in K^*$. By assumption, $\Lambda$ is also not isomorphic to any of algebras $D(\lambda)$, $D(\lambda)^{(1)}$, $D(\lambda)^{(2)}$, $\lambda \in K^*$. Then we conclude by Theorem \ref{D} that $\Lambda$ admits a quotient algebra $\Gamma$ which is a string algebra of non-polynomial growth such that $\Gamma=KQ/L$ and $L$ is an ideal of $KQ$ generated by $\alpha f(\alpha)$ and $A_{\alpha}$, for $\alpha \in Q$. Assume now that $A\cong \Lambda(Q, f, m_{\bullet}, c_{\bullet}, b_{\bullet})=KQ/I'$ where $I'$ is as in the definition of the socle deformed weighted surface algebra. It is easy to see that $I'$ is contained in $L$, thus $\Gamma=KQ/L$ is also a quotient algebra of $A=KQ/I'$.
\end{proof}


In \cite{ES3B} Erdmann and Skowro\'nski showed that the class of indecomposable idempotent algebras of biserial weighted surface algebras coincides with the class of Brauer graph algebras (see e.g. \cite{Rog}). This naturally led to the question about the structure of idempotent algebras of weighted surface algebras and resulted in the introduction of a new class, called the \emph{hybrid algebras} \cite{ES6}. In order to present this class we need to generalize the definition of virtual arrows.

Let $\mathcal{T}$ be a fixed (possibly empty) set of triangles in a biserial quiver $(Q,f)$. An arrow $\alpha$ is then called \emph{virtual} if one of the following holds: $m_{\alpha}n_{\alpha}=1$ and $\alpha, \ba \notin \mathcal{T}$, or $m_{\alpha}n_{\alpha}=2$ and $\ba \in \mathcal{T}$ (see \cite[Definition 3.1]{ES6}). Note that virtual arrows for weighted surface algebras are also virtual in this sense, because they satisfy the second condition.

The \textit {hybrid algebra} $H_{\mathcal{T}}=H_{\mathcal{T}}(Q, f, m_{\bullet}, c_{\bullet})$ associated to $\mathcal{T}$ is the quotient algebra
$H_{\mathcal{T}}=KQ/I$ where $(Q,f)$ satisfies the condition:
\begin{enumerate}[($\ast$)]
\item $m_{\alpha}n_{\alpha}\geq 2$ for any arrow $\alpha$ and $m_{\alpha}n_{\alpha}=1$ is allowed when $\alpha, \ba \notin \mathcal{T}$,
\end{enumerate}
and $I$ is generated by the following elements:
\begin{enumerate}[(i)]
\item $\alpha f(\alpha)-c_{\ba}A_{\ba}$ for $\alpha\in \mathcal{T}$ and $\alpha f(\alpha)$ for $\alpha \notin
    \mathcal{T}$,
\item $\alpha f(\alpha)g(f(\alpha))$ unless $\alpha, \ba \in \mathcal{T}$ and $\ba$ is either virtual or $f(\ba)$ is virtual and $m_{\ba}=1$, $n_{\ba}=3$,
\item $\alpha g(\alpha)f(g(\alpha))$ unless $\alpha, g(\alpha) \in \mathcal{T}$, and $f(\alpha)$ is either virtual or $f^2(\alpha)$ is virtual and $m_{f(\alpha)}=1, n_{f(\alpha)}=3$,
\item $c_{\alpha}B_{\alpha} - c_{\ba}B_{\ba}$ for any arrow $\alpha \in Q$,
\item $B_{\alpha}\alpha \in I$ and $\alpha B_{g(\alpha)}\in I$, for any arrow $\alpha$, if all arrows of $Q$ are virtual.
\end{enumerate}

Let us mention that a few algebras have intentionally been excluded from the class of hybrid algebras due to the fact that they are not symmetric (see \cite{ES6}*{Assumption 3.4}).

Observe that if $\mathcal{T}=\emptyset$ then the algebra $H_{\mathcal{T}}$ is special biserial and symmetric, that is, a Brauer graph algebra. On the other hand, if $\mathcal{T}=Q_1$ then $H_{\mathcal{T}}$ is a weighted surface algebra provided $Q_0$ has at least two vertices \cite{ES4}. Moreover,  the Gabriel quiver $Q_H$ of a hybrid algebra $H$ is obtained from $Q$ by removing the virtual arrows, except when $H$ is local with two virtual loops \cite[Lemma 3.6]{ES6}.

We have the following characterization of hybrid algebras \cite[Theorem 1.1]{ES6}.


\begin{thm} The following assertions hold.
\begin{itemize}
\item[$(1)$] Assume $\Lambda$ is a weighted surface algebra and $e$ is an idempotent of $\Lambda$. Then every block component of $e\Lambda e$ is a hybrid algebra, up to the socle equivalence.
\item[$(2)$] Assume $H$ is a hybrid algebra. Then $H$ is isomorphic to a block component of $e \Lambda e$ for some weighted surface algebra $\Lambda$ and an idempotent $e$ of $\Lambda$.
\end{itemize}
\end{thm}

As a consequence of the above theorem, we get that hybrid algebras are of tame (or finite) representation type and are symmetric, being idempotent algebras of tame symmetric algebras \cite[Lemma 5.7]{ES6}. However, they do not need to be periodic \cite[Section 6]{ES6} (see also Example \ref{ex2}).

\section{Krull-Gabriel dimension of hybrid algebras with triangulation quiver}

As mentioned above, a weighted surface algebra is a hybrid algebra $H_{\mathcal{T}}(Q, f, m_{\bullet}, c_{\bullet})$ such that $\mathcal{T}=Q_1$. In this case the biserial quiver $(Q,f)$ is a triangulation quiver. We now consider a class of non-local hybrid algebras $H_{\mathcal{T}}(Q, f, m_{\bullet}, c_{\bullet})$ for which the underlying quiver $(Q,f)$ is a triangulation quiver, and determine their Krull-Gabriel dimension. In particular, we shall see that the dimension does not depend on the choice of $\mathcal{T}$. Observe that if $(Q,f)$ is a triangulation quiver, then we may choose $Q_1$ as the set of triangles $\mathcal{T}$. 

\begin{prop} \label{hybrid}
Let $H_{\mathcal{T}}=H_{\mathcal{T}}(Q, f, m_{\bullet}, c_{\bullet})$ be a hybrid algebra such that $(Q,f)$ is a triangulation quiver with $|Q_0| \geq 2$ and $m_{\alpha}n_{\alpha} \geq 2$, for any arrow $\alpha$. If a weighted surface algebra $H_{Q_1}=H_{Q_1}(Q, f, m_{\bullet}, c_{\bullet})$ is not isomorphic to any of 
algebras $\Lambda(\lambda)$, $T(\lambda)$, $S(\lambda)$, $D(\lambda)$, $D(\lambda)^{(1)}$, $D(\lambda)^{(2)}$, $\lambda \in K^*$,
then $H_{\mathcal{T}}$ is of non-polynomial growth and $\KG(H_{\mathcal{T}})=\infty$.
\end{prop}


\begin{proof}
If $H_{\mathcal{T}}=H_{\mathcal{T}}(Q, f, m_{\bullet}, c_{\bullet})$ is a hybrid algebra, then $H_{\mathcal{T}}=KQ/I_{\mathcal{T}}$ and $I_{\mathcal{T}}$ satisfies conditions (i)-(v) from the definition of a hybrid algebra. Since we assume that $|Q_0| \geq 2$ and $m_{\alpha}n_{\alpha} \geq 2$, for any arrow $\alpha$, we conclude from the definition of a weighted surface algebra that $H_{Q_1}=H_{Q_1}(Q, f, m_{\bullet}, c_{\bullet})$ is such an algebra with the same triangulation quiver $(Q,f)$, weight function $m_{\bullet}$ and parameter function $c_{\bullet}$ as in $H_{\mathcal{T}}$.
Further, if  $H_{Q_1}$ is not isomorphic to $\Lambda(\lambda)$, $T(\lambda)$, $S(\lambda)$, $D(\lambda)$, $D(\lambda)^{(1)}$, $D(\lambda)^{(2)}$, then by Theorem \ref{D}, there is a quotient algebra $\Gamma$ of $H_{Q_1}$ which is a string algebra of non-polynomial growth. Recall that $\Gamma=KQ/L$, where $L$ is an ideal of the path algebra $K Q$ of $Q$
generated by the elements $\alpha f(\alpha)$ and $A_{\alpha}$, for all arrows $\alpha \in Q_1$. Observe that the ideal $I_{\mathcal{T}}$ is contained in $L$ and thus $\Gamma=KQ/L$ is also a quotient of $H_{\mathcal{T}}=KQ/I_{\mathcal{T}}$, for any set of triangles ${\mathcal{T}}$. Since hybrid algebras are not wild by \cite{ES6}*{Lemma 5.7}, we conclude that $H_{\mathcal{T}}$ is tame of non-polynomial growth. Moreover, we get $\KG (\Gamma)\leq\KG (H_{\mathcal{T}})$ by Lemma \ref{00} and thus $\KG(H_{\mathcal{T}})=\infty$, because $\KG(\Gamma)=\infty$ by Theorem \ref{00000}(\ref{Sch}).
\end{proof}


We give now two examples to illustrate the situation. In the first we present a hybrid algebra with Krull-Gabriel dimension undefined whose quiver is not triangulation quiver. This shows that, in a sense, the converse implication in Proposition \ref{hybrid} does not hold.
 
\begin{example}\label{ex1}
Let $H=H_{\mathcal{T}}(Q, f, m_{\bullet}, c_{\bullet})$ be a hybrid algebra where $(Q,f)$ is of the form 
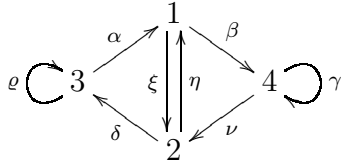
\[
  \xymatrix@R=1pc{
   & 1 \ar[rd]^{\beta} \ar@<-.5ex>[dd]_{\xi}  \\
   3   \ar@(ld,ul)^{\varrho}[]  \ar[ru]^{\alpha}
   && 4   \ar@(ru,dr)^{\gamma}[]  \ar[ld]^{\nu}   \\
   & 2 \ar[lu]^{\delta}  \ar@<-.5ex>[uu]_{\eta}
  }
\]
where $m_{\bullet}, c_{\bullet}$ are trivial weight and parameter functions, and where   $f$ is the permutation with orbits $(\alpha \; \beta \; \nu \; \delta)$, $(\xi \; \eta)$, $(\varrho)$, $(\gamma)$ and $\mathcal{T}=\{\varrho, \;\gamma \}$. Then $(\alpha \;\xi \; \delta \;\varrho)$ and $(\beta \; \gamma \; \nu \;\eta)$ are $g$-orbits of length $4$ and hence there are no virtual arrows in $(Q,f)$. Observe that $(Q,f)$ is not a triangulation quiver. By definition we have the following relations in $H$:
\begin{align*}
 \varrho ^2 &=\alpha \xi \delta ,
 &
 \gamma ^2&=  \nu \eta \beta ,
 &
 \alpha \beta &=0 ,
 &
 \beta \nu &= 0 ,
 &
 \nu \delta& =0,
 &
 \delta \alpha& =0
 \\
 \eta \xi & =0 ,
 &
 \xi \eta & =0,
 &
 \gamma ^2 \nu & =0 ,
 &
 \beta \gamma ^2 & =0 ,
 &
 \varrho ^2 \alpha & =0 ,
 &
 \delta \varrho ^2 & =0
 \end{align*}
\begin{center}
and the commutativity relations $c_{\alpha}B_{\alpha} - c_{\ba}B_{\ba}$ for any arrow $\alpha \in Q$.
\end{center}
We repeat the construction of the string quotient algebra $\Gamma$ of $H$. Namely, assume that $\Gamma=KQ/L$ where $L$ is the ideal generated by relations of the form $\alpha f(\alpha)$ and $A_{\alpha}$. Hence $L$ is generated by the paths of length two: $\varrho ^2$, $ \gamma ^2 $, $ \alpha \beta $, $ \beta \nu $, $ \nu \delta$, $\delta \alpha$, $\eta \xi$, $\xi \eta$ and all paths of length three. Then $u=\xi \nu^{-1} \gamma \beta^{-1}$ and $w=\xi \delta \varrho^{-1} \alpha \xi \nu^{-1} \gamma \beta^{-1}$ are primitive walks of the bound quiver $(Q,L)$ such that $uw$ and $wu$ are also primitive walks and thus we conclude that $\Gamma$ is of non-polynomial growth (see e.g. the proof of \cite[Lemma 1]{Sk0}). Again, by Theorem \ref{00000}(\ref{Sch}), we have $\KG(\Gamma)=\infty$ and by Lemma \ref{0} we obtain that $\KG(H)=\infty$.  
\end{example}

The second example shows a hybrid algebra associated to a non-empty set of triangles which has finite Krull-Gabriel dimension.

\begin{example} \label{ex2}
Consider a tetrahedral algebra $\Lambda(\lambda)=KQ(\Lambda(\lambda))/I$, $\lambda \in K^*$, given as in Section 3. By $e_i$ for $i \in Q_0$ we denote the trivial path at vertex $i$. Let $e=1-e_1$ be an idempotent of algebra $\Lambda(\lambda)$. Then $$H(\lambda)=e\Lambda(\lambda)e=eKQ(\Lambda(\lambda))e/eIe$$ is isomorphic with the quotient path algebra $K\widetilde{Q}/\widetilde{I}$ of a biserial quiver $(\widetilde{Q},f')$, where $\widetilde{Q}$ is of the form:
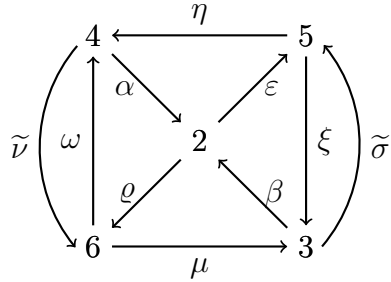
\[
\begin{tikzpicture}
[scale=.7]

\node (2) at (0,-2) {$2$};
\node (3) at (2,-4) {$3$};
\node (4) at (-2,0) {$4$};
\node (5) at (2,0) {$5$};
\node (6) at (-2,-4) {$6$};
\coordinate (2) at (0,-2);
\coordinate (3) at (2,-4);
\coordinate (4) at (-2,0);
\coordinate (5) at (2,0);
\coordinate (6) at (-2,-4);
\node (2) at (0,-2) {$2$};
\node (3) at (2,-4) {$3$};
\node (4) at (-2,0) {$4$};
\node (5) at (2,0) {$5$};
\node (6) at (-2,-4) {$6$};
\draw[->,thick] (-2.3,-0.2) arc [start angle=140, delta angle=79, radius=3cm] node[midway,left] {$\widetilde{\nu}$};
\draw[->,thick] (2.3,-4) arc [start angle=-40, delta angle=79, radius=3cm] node[midway,right] {$\widetilde{\sigma}$};
\draw[->,thick]
(2) edge node [right] {$\varepsilon$} (5)
(2) edge node [left] {$\varrho$} (6)
(3) edge node [right] {$\beta$} (2)
(4) edge node [left] {$\alpha$} (2)
(5) edge node [right] {$\xi$} (3)
(5) edge node [above] {$\eta$} (4)
(6) edge node [left] {$\omega$} (4)
(6) edge node [below] {$\mu$} (3)
;
\end{tikzpicture}
\]
and permutation $f'$ has orbits $(\alpha\; \varrho\; \omega)$, $(\xi\; \beta\; \varepsilon)$, inherited from the action of $f$ on $Q(\Lambda(\lambda))$, and an orbit $(\eta\; \widetilde{\nu}\; \mu\; \widetilde{\sigma})$ of length $4$. The generators of $\widetilde{I}$ come from the generators of $I$. Namely, $\widetilde{I}$ is generated by commutativity relations:
\begin{align*}
  \alpha \varrho &= \widetilde{\nu} ,
  &
  \varrho \omega &= \lambda \varepsilon \eta ,
  &
  \omega \alpha &= \mu \beta ,
  &
  \beta \varepsilon &= \widetilde{\sigma} ,
  &
  \varepsilon \xi &= \varrho \mu ,
  &
  \xi \beta &= \lambda \eta \alpha ,
\end{align*}
\begin{center}
$c_{\theta}B_{\theta}-c_{\bar{\theta}}B_{\bar{\theta}} \qquad $\text{for all arrows} $\theta \in \widetilde{Q},$
\end{center}
and zero-relations of the form:
\begin{align*}
  \eta \widetilde{\nu} &=0,
  &
  \widetilde{\sigma}\eta &=0,
  &
  \widetilde{\nu} \mu &=0,
  &
  \mu  \widetilde{\sigma} &=0,
\end{align*}
\begin{center}
$\theta f'(\theta)g'(f'(\theta)) =0 \quad $\text{and} $\quad   \theta g'(\theta)f'(g'(\theta))=0 \quad $\text{for all arrows} $\theta \in \widetilde{Q}\; (g'=\overline{f'}).$
\end{center}
Note that $H(\lambda)$ is a hybrid algebra $H_{\mathcal{T}}(\widetilde{Q}, f', m'_{\bullet}, c_{\bullet})$ with the set of triangles $\mathcal{T}=\{(\alpha\; \varrho\; \omega), (\xi\; \beta\; \varepsilon)\}$, a trivial weight function $m'_{\bullet}$ and parameter function $c'_{\bullet}$ inherited from $\Lambda(\lambda)$, that is, $c'_{\alpha}=c'_{\varepsilon}=c'_{\eta}=\lambda$ and is trivial everywhere else. Since $f'$ is not of order $3$, the quiver $(Q,f')$ is not triangulation quiver and we cannot apply Proposition \ref{hybrid} to compute Krull-Gabriel dimension of $H(\lambda)$. In order to do so, we need an admissible presentation of this algebra.
Note that the arrows $\widetilde{\nu}$ and $\widetilde{\sigma}$ are virtual, thus the Gabriel quiver $Q_{H(\lambda)}$ of $H(\lambda)$ is of the form:
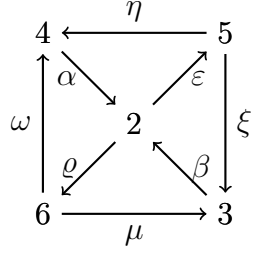
\[
\begin{tikzpicture}
[scale=.6]

\node (2) at (0,-2) {$2$};
\node (3) at (2,-4) {$3$};
\node (4) at (-2,0) {$4$};
\node (5) at (2,0) {$5$};
\node (6) at (-2,-4) {$6$};
\coordinate (2) at (0,-2);
\coordinate (3) at (2,-4);
\coordinate (4) at (-2,0);
\coordinate (5) at (2,0);
\coordinate (6) at (-2,-4);
\node (2) at (0,-2) {$2$};
\node (3) at (2,-4) {$3$};
\node (4) at (-2,0) {$4$};
\node (5) at (2,0) {$5$};
\node (6) at (-2,-4) {$6$};
\draw[->,thick]
(2) edge node [right] {$\varepsilon$} (5)
(2) edge node [left] {$\varrho$} (6)
(3) edge node [right] {$\beta$} (2)
(4) edge node [left] {$\alpha$} (2)
(5) edge node [right] {$\xi$} (3)
(5) edge node [above] {$\eta$} (4)
(6) edge node [left] {$\omega$} (4)
(6) edge node [below] {$\mu$} (3)
;
\end{tikzpicture}
\]
and admissible ideal $I_{H(\lambda)}$ is generated by:
\[
  \varrho \omega - \lambda \varepsilon \eta ,
  \quad
  \omega \alpha - \mu \beta ,
  \quad
  \varepsilon \xi - \varrho \mu ,
  \quad
  \xi \beta - \lambda \eta \alpha ,
\]
\[
 \alpha \varrho \mu ,\quad
 \alpha \varepsilon \xi ,\quad 
 \beta \varepsilon \eta ,\quad 
 \beta \varrho \omega ,\quad 
 \omega \alpha \varepsilon ,\quad 
 \mu \beta \varepsilon  ,\quad 
 \eta \alpha \varrho ,\quad 
 \xi \beta \varrho .
\]

\noindent Observe that $H(\lambda)$ is a trivial extension of the one-point extension algebra $A[M_{\lambda}]$ of a path algebra $A=K\Delta$ of a quiver $\Delta$ of Euclidean type $\widetilde{A}_3$
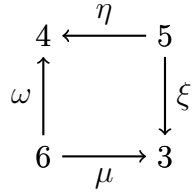
\[
\begin{tikzpicture}
[scale=.4]
\node (3) at (2,-4) {$3$};
\node (4) at (-2,0) {$4$};
\node (5) at (2,0) {$5$};
\node (6) at (-2,-4) {$6$};
\coordinate (3) at (2,-4);
\coordinate (4) at (-2,0);
\coordinate (5) at (2,0);
\coordinate (6) at (-2,-4);
\node (3) at (2,-4) {$3$};
\node (4) at (-2,0) {$4$};
\node (5) at (2,0) {$5$};
\node (6) at (-2,-4) {$6$};
\draw[->,thick]
(5) edge node [right] {$\xi$} (3)
(5) edge node [above] {$\eta$} (4)
(6) edge node [left] {$\omega$} (4)
(6) edge node [below] {$\mu$} (3)
;
\end{tikzpicture}
\]
by an indecomposable module $M_{\lambda}$ 
\[
\begin{tikzpicture}
[scale=.4]
\node (3) at (2,-4) {$K$};
\node (4) at (-2,0) {$K$};
\node (5) at (2,0) {$K$};
\node (6) at (-2,-4) {$K$};
\coordinate (3) at (2,-4);
\coordinate (4) at (-2,0);
\coordinate (5) at (2,0);
\coordinate (6) at (-2,-4);
\node (3) at (2,-4) {$K$};
\node (4) at (-2,0) {$K$};
\node (5) at (2,0) {$K$};
\node (6) at (-2,-4) {$K$};
\draw[->,thick]
(5) edge node [right] {$1$} (3)
(5) edge node [above] {$\lambda$} (4)
(6) edge node [left] {$1$} (4)
(6) edge node [below] {$1$} (3)
;
\end{tikzpicture}
\]
lying on the mouth of a standard homogenous tube in $\Gamma_A$, see \cite{SiSk3}. Then $H(\lambda)$ is a symmetric algebra of Euclidean type \cite{BoSk}, and hence a standard domestic selfinjective algebra of infinite representation type \cite{Sk1} (see also \cite{Sk2}). Then, following Theorem \ref{00000}(\ref{P}), we have $\KG(H(\lambda))=2$. Note that $H(\lambda)$ is not periodic by \cite{BES}.
\end{example}

Now we are able to present the main result for the weighted surface algebras.

\begin{thm} \label{wsa}
Assume that $\Lambda=\Lambda(Q, f, m_{\bullet}, c_{\bullet})$ is a weighted surface algebra. Then $\KG(\Lambda)=\infty$.
\end{thm}
\begin{proof}
We have the following cases to consider.

(1)  Assume that $\Lambda$ is not isomorphic to a weighted surface algebra $D(\lambda)^{(1)}$, $D(\lambda)^{(2)}$, $D(\lambda)$, $\Lambda(\lambda)$, $T(\lambda)$, $S(\lambda)$, for any $\lambda \in K^*$. Then $\Lambda$ is a hybrid algebra $H_{\mathcal{T}}(Q, f, m_{\bullet}, c_{\bullet})$ with $\mathcal{T}=Q_1$, the same triangulation quiver $(Q, f)$, weight function $m_{\bullet}$ and parameter function $c_{\bullet}$. Hence $\KG(\Lambda)=\infty$ by Proposition \ref{hybrid}.

(2) Let $\Lambda$ be isomorphic to $D(\lambda)^{(1)}$ or $D(\lambda)^{(2)}$, for $\lambda\in K^*$. We use the arguments given in the proof of \cite[Lemma 6.9]{ES4}, in which it was shown that $D(\lambda)^{(1)}$, $D(\lambda)^{(2)}$ are of non-polynomial growth. Namely, let $(Q^{(i)}, I^{(i)})$ be the bound quiver of the algebra $D(\lambda)^{(i)}$, for $i \in \{1,2\}$. Consider the quotient algebra
$A^{(i)}$ of $D(\lambda)^{(i)}$ given by the Gabriel quiver $Q^{(i)}$ and the ideal of relations generated by zero-relations $\omega=0$ and
$v=0$ for any commutativity relation $\omega-v$ in $I^{(i)}$. By the proof of \cite[Lemma 6.9]{ES4} we know that $A^{(i)}$ admits
the Galois covering $F^{(i)}\colon R \rightarrow A^{(i)}$ by a locally bounded category $R$ such that $A^{(i)}=R/G^{(i)}$ for some free
abelian group $G^{(i)}$, $i \in \{1,2\}$. Further, $R$ contains a finite full convex subcategory $B$, which is a minimal pg-critical algebra listed in \cite[Theorem 3.2]{NorSko}. Hence we have the following diagram:
      \vspace{-0.3cm}\[
    \begin{tikzpicture}
    [scale=.4]
    \node (D) at (4.6,4) {$D(\lambda)^{(i)}$};
    \node (R) at (-4,0) {$R$};
    \node (A) at (2,0) {$R/G^{(i)}$};
    \node (A') at (4.6,0) {$=A^{(i)}$};
    \node (B) at (-4,-3.4) {$B$};
    \draw[->,thick]
    (B) edge node [left] {$\iota$} (R)
    (R) edge node [above] {$F$} (A)
    (D) edge node [right] {$p$} (A')
    ;
    \end{tikzpicture}\vspace{-0.3cm}
    \]
    where $p \colon D(\lambda)^{(i)} \rightarrow A^{(i)}$  is a surjection of $K$-algebras and $\iota \colon B \rightarrow R$ is an embedding of locally bounded $K$-categories. Since $B$ is pg-critical,  following Theorem \ref{00000}(\ref{KP}) we have $\KG(B) =\infty$. By Lemma \ref{00}(a) we conclude that $\KG(B) \leq \KG(R)$, and therefore 
    $\KG(R) =\infty$. Applying Theorem \ref{A} to the Galois covering $F^{(i)}\colon R \rightarrow A^{(i)}$ we have $\KG(A^{(i)})=\infty$ and then, by Lemma \ref{00}(b), $\KG(\Lambda)=\KG(D(\lambda)^{(i)})=\infty$ for $i \in \{1,2\}$ and $\lambda \in K^*$.

(3) Assume that $\Lambda$ is isomorphic to $D(\lambda)$, $\Lambda(\lambda)$, $T(\lambda)$ or $S(\lambda)$ for $\lambda\in K^*$.
    We have the following characterization from \cite{ES4}*{Section 3}:
    \begin{itemize}
    \item[(i)]  non-singular tetrahedral algebras $\Lambda(\lambda)$, $\lambda \in K \backslash \{0,1\}$, are isomorphic to the trivial extension algebras $\T(B(\lambda))$ of some tubular algebras $B(\lambda)$ of type $(2,2,2,2)$ (see \cite{R} for tubular algebras),
    \item[(ii)] non-singular spherical algebras $S(\lambda)$, $\lambda \in K \backslash \{0,1\}$, are isomorphic to the trivial extension algebras $\T(C(\lambda))$ of some tubular algebras $C(\lambda)$ of type $(2,2,2,2)$,
    \item[(iii)] the singular tetrahedral algebra $\Lambda(1)$ is isomorphic to  the trivial extension algebra $\T(B(1))$ of some pg-critical algebra $B(1)$,
    \item[(iv)] the singular spherical algebra $S(1)$ is isomorphic to the trivial extension algebra $\T(C(1))$ of some pg-critical algebra $C(1)$,
    \item[(v)]  disc algebras $D(\lambda)$ are isomorphic to orbit algebras $\Lambda(\lambda)/\mathbb{Z}_3$, for any $\lambda \in K^*$,
    \item[(vi)] triangle algebras $T(\lambda)$ are isomorphic to  orbit algebras $S(\lambda)/ \mathbb{Z}_2$, for any $\lambda \in K^*$.
    \end{itemize} 
In (i) and (ii), a weighted surface algebra $\Lambda$ is isomorphic to the orbit algebra $\widehat{B}/G$ of the repetitive category $\widehat{B}$ of a tubular algebra $B$ with respect to the infinite cyclic group $G=(\nu_{\widehat{B}})$. Thus it is a standard selfinjective non-domestic algebra of polynomial growth by \cite{Sk1}*{Theorem}. Then, following Theorem \ref{00000}(\ref{P}), we obtain $\KG(\Lambda)=\infty$. Further, in case of (iii) and (iv), a weighted surface algebra $\Lambda$ is isomorphic to the orbit algebra $\widehat{B}/G$ of the repetitive category $\widehat{B}$ of some pg-critical algebra $B$ and the infinite cyclic group $G=(\nu_{\widehat{B}})$. We consider an embedding of finite convex subcategory $B$ into the category $\widehat{B}$. Then by Lemma \ref{00} we get $\KG(B)\leq \KG(\widehat{B})$ which yields $\KG(\widehat{B})=\infty$ because $B$ is pg-critical, see Theorem \ref{00000}(\ref{KP}). Hence by Corollary \ref{000}, we obtain that $\KG(\Lambda)=\KG(\T(B))=\infty$. Finally, in (v) and (vi) a weighted surface algebra $\Lambda$ is isomorphic to the orbit algebra $\Lambda'/H$, where $\Lambda'$ is isomorphic to $\Lambda(\lambda)$ or $S(\lambda)$, for any $\lambda \in K^*$, and $H$ is a finite cyclic group. By applying Theorem \ref{A} to the Galois covering $F \colon\,  \Lambda' \rightarrow \Lambda'/H$, we obtain $\KG(\Lambda)=\KG(\Lambda'/H)=\infty$, because $\KG(\Lambda')=\infty$ by (i)-(iv).
\end{proof}

\begin{rmk}
The Galois covering $F^{(i)}\colon R \rightarrow A^{(i)}$ from the proof of (2) has a torsion-free covering group $G^{(i)}$, $i \in \{1,2\}$, but $R$ is not locally support-finite (see the proof of \cite[Lemma 6.9]{ES4}). Also in (3)(iii) and (iv), in the Galois covering $F\colon \widehat{B} \rightarrow \T(B)$ the repetitive category $\widehat{B}$ of the pg-critical algebra $B$ is not locally support-finite, see \cite{AS0}*{Theorem (B)}. In turn, 
in (3)(v) and (vi) we have the Galois covering $F \colon  \Lambda' \negthickspace \rightarrow \Lambda'/H$ in which the opposite situation occurs. Namely, $\Lambda'$ is finite, hence locally support-finite, and the covering group is not torsion-free. Therefore, in all of these cases we cannot apply Theorem \ref{B}, but only the general result of Theorem \ref{A}.
\end{rmk}

\begin{rmk}
Assume that $\lambda \in K \backslash \{0,1\}$. By the characterization of selfinjective tubular algebras given by Bia\l kowski and Skowro\'nski in \cite{BiS}, we conclude that the disc algebras $D(\lambda)\cong\Lambda(\lambda)/\mathbb{Z}_3$ and the triangle algebras $T(\lambda)\cong S(\lambda)/ \mathbb{Z}_2$ are isomorphic to orbit algebras $\widehat{B}/G$ of the repetitive category $\widehat{B}$ of a tubular algebra $B$ of type  $(2,2,2,2)$ and some admissible infinite cyclic group $G$ of automorphisms of $\widehat{B}$. Thus they are standard selfinjective non-domestic algebras of polynomial growth by \cite{Sk1}. Therefore, to determine the Krull-Gabriel dimension of $D(\lambda)$ and $T(\lambda)$, for $\lambda \in K \backslash \{0,1\}$, we may also apply Theorem \ref{00000}(\ref{P}). Nevertheless, the case of $\lambda=1$ requires the arguments given in the proof of Theorem \ref{wsa}.
\end{rmk}

We can now easily conclude as follows.

\begin{thm}\label{support}
The class of weighted surface algebras supports Prest's conjecture.
\end{thm}
\begin{proof}
By Theorem \ref{C} any weighted surface algebra is tame non-domestic and by Theorem \ref{wsa} its Krull-Gabriel dimension is undefined, hence not finite.
\end{proof}

Applying Proposition \ref{E}, we obtain the following generalisation of Theorem \ref{wsa}.
\begin{cor} \label{soc}
Let $A$ be a symmetric algebra which is socle equivalent to a weighted surface algebra $\Lambda$. If $\Lambda$ is not isomorphic to $D(\lambda)^{(1)}$, $D(\lambda)^{(2)}$, $D(\lambda)$, for $\lambda \in K^*$, then $\KG(A)=\infty$. 
\end{cor}
\begin{proof}
  If $A$ is isomorphic to a weighted surface algebra, then the statement follows from Theorem \ref{wsa}. Assume now that $A$ is socle equivalent, but not isomorphic to a weighted surface algebra $\Lambda$. Then by Proposition \ref{E}, there is a string quotient algebra $\Gamma$ of $A$ which is of non-polynomial growth. As before we conclude that $\KG(A)=\infty$ by Lemma \ref{0} and Theorem \ref{00000}(\ref{Sch}).
\end{proof}


Following \cite{ES3}, an algebra $A$ is of \emph{generalized quaternion type} if $A$ is representation-infinite tame symmetric and every simple module in $\mod (A)$ is periodic of period $4$. These algebras have their origin in the study of blocks of group algebras with generalized quaternion defect groups \cite{E} and, as it was recently proved in \cite{ES3}, they are closely related to weighted surface algebras. We finish this section with determining Krull-Gabriel dimension of some particular class of algebras of generalized quaternion type. 

\begin{rmk}\label{rem_fin1} In \cite{ES2} Erdmann and Skowro\'nski defined the \emph{higher tetrahedral algebras} $\Lambda(m, \lambda)$, a new class of symmetric algebras with underlying triangulation quiver associated to the tetrahedron (as in Section 3) and a set of relations depending on parameter $\lambda \in K$ and natural number $m\geq 2$. If $m=1$, then the set of relations is exactly the same as for tetrahedral algebras. However, unlike tetrahedral algebras, higher tetrahedral algebras are tame of non-polynomial growth, for any $\lambda \in K$. In fact, they admit a pg-critical algebra as a quotient algebra, see the proof of \cite[Proposition 5.6]{ES2}. Moreover, higher tetrahedral algebras are periodic of period $4$ and so they share essential properties with weighted surface algebras (see Theorem \ref{C}). Actually, these two classes of algebras play a crucial role in a characterization of tame symmetric periodic algebras of period $4$. 
\end{rmk}

It was shown in \cite{ES3} that an algebra $A$ is of generalized quaternion type with $2$-regular Gabriel quiver having at least $3$ vertices if and only if $A$ is socle equivalent to a weighted surface algebra, different from the singular tetrahedral algebra $\Lambda(1)$, or is isomorphic to a higher tetrahedral algebra $\Lambda(m, \lambda)$ for some $\lambda \in K^*$. In particular, we get that $A$ is periodic of period $4$. This characterization leads us to the following conclusion.

\begin{cor} \label{quaternion}
Let $A$ be an algebra of generalized quaternion type which has $2$-regular Gabriel quiver with at least $3$ vertices. Assume that $A$ is not socle equivalent with $D(\lambda)^{(2)}$, for any $\lambda\in K^*$. Then $\KG(A)=\infty$.
\end{cor}
\begin{proof}
We have three cases to consider. Namely, if $A$ is isomorphic to a weighted surface algebra, then we apply Theorem \ref{wsa}. If $A$ is socle equivalent but non-isomorphic to a weighted surface algebra $\Lambda$ (different from the singular tetrahedral algebra), then the Grothendieck group $K_0(\Lambda)$ of $\Lambda$ is also of rank at least $3$. Therefore, $\Lambda$ is not isomorphic to disc algebras $D(\lambda)$ and $D(\lambda)^{(1)}$, for any $\lambda \in K^*$, because these algebras have the Grothendieck groups of rank 2. By assumption, $\Lambda$ is not isomorphic to $D(\lambda)^{(2)}$ either, for any $\lambda \in K^*$, and hence by Corollary \ref{soc} we get that $\KG(A)=\infty$. Finally, if $A$ is isomorphic to a higher tetrahedral algebra $\Lambda(m,\lambda)$, then by Lemma \ref{00}(b) and Theorem \ref{00000}(\ref{KP}) we have $\infty=\KG(\Gamma)\leq \KG(\Lambda(m, \lambda))=\KG(A)$, where $\Gamma$ is the quotient algebra of $A$ which is pg-critical, see Remark \ref{rem_fin1}.
\end{proof}

\begin{rmk} Observe that in view of Corollary \ref{soc}, it would be interesting to know if a symmetric algebra which is socle equivalent with $D(\lambda)^{(1)}$, $D(\lambda)^{(2)}$ or $D(\lambda)$, for $\lambda \in K^*$, has the Krull-Gabriel dimension undefined. Similarly, in view of Corollary \ref{quaternion}, it would be interesting to know if an algebra of generalized quaternion type having $2$-regular Gabriel quiver with at least $3$ vertices which is socle equivalent with $D(\lambda)^{(2)}$ has the Krull-Gabriel dimension undefined. Positive answers to these questions allow to formulate \ref{soc} and \ref{quaternion} more generally. These problems are left for further research.
\end{rmk}

\begin{rmk}\label{rem_fin2} Krause shows in \cite{Kr2} that algebras which are stable equivalent have the same Krull-Gabriel dimension. Since derived equivalence implies stable equivalence in the class of selfinjective algebras \cite{Ric}, we conclude that algebras which are derived equivalent to any of algebras considered in \ref{hybrid}, \ref{wsa}, \ref{soc} and \ref{quaternion} have the Krull-Gabriel dimension undefined. For example, in \cite{ES7} Erdmann and Skowro\'nski introduce the class of \emph{higher spherical algebras}. It follows from \cite{ES7}*{Theorem 3} that a higher spherical algebra $S(m,\lambda)$ where $m>2$ and $\lambda\in K^*$ is derived equivalent with the higher tetrahedral algebra $\Lambda(m,\lambda)$. Thus these algebras have the Krull-Gabriel dimension undefined as well. 
\end{rmk}

\end{document}